\documentclass[11pt]{article}
\usepackage[margin=1in]{geometry}
\usepackage{float,longtable}
\usepackage{enumitem,mathtools,microtype,setspace}
\usepackage[dvipsnames]{xcolor}
\definecolor{BasisMassLink}{RGB}{22,67,114}
\definecolor{BasisMassCite}{RGB}{35,102,78}
\definecolor{BasisMassURL}{RGB}{0,91,150}
\usepackage[
  colorlinks=true,
  linkcolor=BasisMassLink,
  citecolor=BasisMassCite,
  urlcolor=BasisMassURL,
  linktoc=all
]{hyperref}
\makeatletter

\RequirePackage{amsmath}
\RequirePackage{amssymb}
\RequirePackage{amsthm}
\RequirePackage{bm}
\usepackage{bbm}
\RequirePackage{url}
\usepackage{natbib}
\usepackage{multirow}
\usepackage{graphicx}
\usepackage{subfigure}
\usepackage{makecell}
\usepackage{booktabs}
\usepackage{array}
\usepackage{algorithm}
\usepackage{algorithmic}
\usepackage{dsfont}
\RequirePackage[dvipsnames]{xcolor}
\RequirePackage[capitalise,noabbrev,nameinlink]{cleveref}

\newcommand{\cD}{\mathcal{D}}

\newcommand{\cM}{\mathcal{M}}

\newtheoremstyle{mytheoremstyle} 
    {\topsep}                    
    {\topsep}                    
    {\normalfont}                   
    {}                           
    {\bfseries}                   
    {.}                          
    {.5em}                       
    {}  

\theoremstyle{mytheoremstyle}
\theoremstyle{plain}
\crefname{assumption}{Assumption}{Assumptions}
\crefname{theorem}{Theorem}{Theorems}
\Crefname{theorem}{Theorem}{Theorems}
\crefname{lemma}{Lemma}{Lemmas}
\Crefname{lemma}{Lemma}{Lemmas}
\crefname{proposition}{Proposition}{Propositions}
\Crefname{proposition}{Proposition}{Propositions}
\crefname{definition}{Definition}{Definitions}
\Crefname{definition}{Definition}{Definitions}
\crefname{corollary}{Corollary}{Corollaries}
\Crefname{corollary}{Corollary}{Corollaries}
\crefname{remark}{Remark}{Remarks}
\Crefname{remark}{Remark}{Remarks}
\crefname{example}{Example}{Examples}
\Crefname{example}{Example}{Examples}
\crefname{conjecture}{Conjecture}{Conjectures}
\Crefname{conjecture}{Conjecture}{Conjectures}
\crefname{fact}{Fact}{Facts}
\Crefname{fact}{Fact}{Facts}
\crefname{claim}{Claim}{Claims}
\Crefname{claim}{Claim}{Claims}
\Crefname{assumption}{Assumption}{Assumptions}
\crefname{condition}{Condition}{Conditions}
\Crefname{condition}{Condition}{Conditions}
\crefname{equation}{Equation}{Equations}
\Crefname{equation}{Equation}{Equations}
\crefname{table}{Table}{Tables}
\Crefname{table}{Table}{Tables}
\crefname{section}{Section}{Sections}
\Crefname{section}{Section}{Sections}
\crefname{subsection}{Section}{Sections}
\Crefname{subsection}{Section}{Sections}

\ifx\BlackBox\undefined
\newcommand{\BlackBox}{\rule{1.5ex}{1.5ex}}  
\fi

\ifx\QED\undefined
\def\QED{~\rule[-1pt]{5pt}{5pt}\par\medskip}
\fi

\ifx\proof\undefined
\newenvironment{proof}{\par\noindent{\bf Proof\ }}{\hfill\BlackBox\\[2mm]}
\fi
\ifx\solution\undefined
\newenvironment{solution}{\par\noindent{\bf Solution\ }}{\hfill\BlackBox\\[2mm]}
\fi

\ifx\theorem\undefined
\theoremstyle{plain}
\newtheorem{theorem}{Theorem}
\fi
\ifx\example\undefined
\theoremstyle{definition}
\newtheorem{example}[theorem]{Example}
\fi
\ifx\property\undefined
\theoremstyle{plain}
\newtheorem{property}{Property}
\fi
\ifx\lemma\undefined
\theoremstyle{plain}
\newtheorem{lemma}[theorem]{Lemma}
\fi
\ifx\proposition\undefined
\theoremstyle{plain}
\newtheorem{proposition}[theorem]{Proposition}
\fi
\ifx\remark\undefined
\theoremstyle{remark}
\newtheorem{remark}[theorem]{Remark}
\fi
\ifx\corollary\undefined
\theoremstyle{plain}
\newtheorem{corollary}[theorem]{Corollary}
\fi
\ifx\definition\undefined
\theoremstyle{definition}
\newtheorem{definition}[theorem]{Definition}
\fi
\ifx\conjecture\undefined
\theoremstyle{plain}
\newtheorem{conjecture}[theorem]{Conjecture}
\fi
\ifx\fact\undefined
\theoremstyle{plain}
\newtheorem{fact}[theorem]{Fact}
\fi
\ifx\claim\undefined
\theoremstyle{plain}
\newtheorem{claim}[theorem]{Claim}
\fi
\ifx\assumption\undefined
\theoremstyle{definition}
\newtheorem{assumption}[theorem]{Assumption}
\fi
\ifx\condition\undefined
\theoremstyle{definition}
\newtheorem{condition}[theorem]{Condition}
\fi

\let\orig@example\example \let\endorig@example\endexample
\renewenvironment{example}{\crefalias{theorem}{example}\orig@example}{\endorig@example}

\let\orig@lemma\lemma \let\endorig@lemma\endlemma
\renewenvironment{lemma}{\crefalias{theorem}{lemma}\orig@lemma}{\endorig@lemma}

\let\orig@proposition\proposition \let\endorig@proposition\endproposition
\renewenvironment{proposition}{\crefalias{theorem}{proposition}\orig@proposition}{\endorig@proposition}

\let\orig@remark\remark \let\endorig@remark\endremark
\renewenvironment{remark}{\crefalias{theorem}{remark}\orig@remark}{\endorig@remark}

\let\orig@corollary\corollary \let\endorig@corollary\endcorollary
\renewenvironment{corollary}{\crefalias{theorem}{corollary}\orig@corollary}{\endorig@corollary}

\let\orig@definition\definition \let\endorig@definition\enddefinition
\renewenvironment{definition}{\crefalias{theorem}{definition}\orig@definition}{\endorig@definition}

\let\orig@conjecture\conjecture \let\endorig@conjecture\endconjecture
\renewenvironment{conjecture}{\crefalias{theorem}{conjecture}\orig@conjecture}{\endorig@conjecture}

\let\orig@fact\fact \let\endorig@fact\endfact
\renewenvironment{fact}{\crefalias{theorem}{fact}\orig@fact}{\endorig@fact}

\let\orig@claim\claim \let\endorig@claim\endclaim
\renewenvironment{claim}{\crefalias{theorem}{claim}\orig@claim}{\endorig@claim}

\let\orig@assumption\assumption \let\endorig@assumption\endassumption
\renewenvironment{assumption}{\crefalias{theorem}{assumption}\orig@assumption}{\endorig@assumption}

\let\orig@condition\condition \let\endorig@condition\endcondition
\renewenvironment{condition}{\crefalias{theorem}{condition}\orig@condition}{\endorig@condition}

\numberwithin{equation}{section}
\numberwithin{theorem}{section}

\makeatother
\numberwithin{theorem}{section}

\crefname{equation}{Equation}{Equations}
\Crefname{equation}{Equation}{Equations}
\crefname{section}{Section}{Sections}
\Crefname{section}{Section}{Sections}
\crefname{subsection}{Section}{Sections}
\Crefname{subsection}{Section}{Sections}
\crefname{subsubsection}{Section}{Sections}
\Crefname{subsubsection}{Section}{Sections}
\crefname{table}{Table}{Tables}
\Crefname{table}{Table}{Tables}
\crefname{appendix}{Appendix}{Appendices}
\Crefname{appendix}{Appendix}{Appendices}

\usepackage{needspace}
\usepackage{etoolbox}
\pretocmd{\section}{\Needspace{8\baselineskip}}{}{}
\pretocmd{\subsection}{\Needspace{6\baselineskip}}{}{}

\newcommand{\DR}{\operatorname{DR}}
\newcommand{\bbE}{\mathbb E}
\newcommand{\bbP}{\mathbb P}
\newcommand{\St}{\mathrm V}
\newcommand{\Gram}{\operatorname{Gram}}

\title{Sharp Minimax Gram-Determinant Asymptotics\\
for Labelled Unit Decompositions of the Identity\\
and Recent Developments on Makai's Conjecture}
\author{Haozhou Zhang\\
Department of Mathematics and Statistics\\
University of Idaho\\
\href{mailto:zhan3063@vandals.uidaho.edu}{\texttt{zhan3063@vandals.uidaho.edu}}}
\date{9 October 2026}
\hypersetup{
  pdftitle={Sharp Minimax Gram-Determinant Asymptotics for Labelled Unit Decompositions of the Identity and Recent Developments on Makai's Conjecture},
  pdfauthor={Haozhou Zhang},
  pdfsubject={Haar--Parseval asymptotics, minimax Gram determinants, Dvoretzky--Rogers counterexamples, and Makai's conjecture from strong slicing},
  pdfkeywords={Dvoretzky--Rogers lemma, Parseval frame, Stiefel manifold,
               Gram determinant, decomposition of the identity, random matrix,
               Makai's conjecture, strong slicing, isotropic constant},
  pdfdisplaydoctitle=true,
  pdflang={en-GB}
}

\begin{document}
\maketitle

\begin{abstract}
  We show that González Merino and Schymura's Conjecture~4.8, a sufficient
  condition for Makai's reverse isodiametric conjecture, fails in all
  sufficiently large dimensions.  Its three-dimensional case follows from
  Aliev's proof.  Pe{\l}czy{\'n}ski and Szarek's construction for enclosing
  parallelepipeds implies this failure.  Using random frames, we also
  determine the sharp exponential determinant guarantee for every fixed
  ratio greater than one between the number of vectors and the dimension.
\end{abstract}

\medskip
\noindent\textbf{Note added on 9 October 2026.}
We note that OpenAI's strong-slicing result, released as item~101
\citep[Theorem~1.1]{OpenAI2026StrongSlicing}, implies Makai's reverse
isodiametric conjecture, including its simplex equality case, as shown in
\cref{app:makai}.  The author has strong confidence in OpenAI's proof, but
does not assume responsibility for the correctness of that external
result.  Following the covariance-normalisation argument given in the
appendix, we establish this implication from their stated theorem.
We also observe that arguments
similar to those used in OpenAI's proof may yield a direct proof of
Makai's conjecture, potentially with simplifications, but the author does
not intend to pursue such an attempt as of current.  The main results of this paper do
not depend on that external input.



\section{Introduction}\label{sec:introduction}

Makai's reverse isodiametric conjecture \citep{Makai1978} asserts that
for every convex body \(K\subseteq\mathbb R^n\) there is an invertible
linear map \(A\) such that
\begin{equation}\label{eq:Makai-conjecture}
 \operatorname{idq}(AK):=
 \frac{\operatorname{vol}_n(AK)}{\operatorname{diam}(AK)^n}
 \ge \frac{\sqrt{n+1}}{n!\,2^{n/2}}.
\end{equation}
The right-hand side is the isodiametric quotient of a regular simplex.
Thus the conjecture asks whether every convex body has a linear image
whose isodiametric quotient is at least that of a regular simplex.

While studying this conjecture, we came across the stronger
simplex-selection Conjecture~4.8 of
\citet{GonzalezMerinoSchymura2020}.  A unit decomposition of the identity
in \(\mathbb R^n\) consists of
\begin{equation}\label{eq:identity-decomposition}
 u_1,\ldots,u_m\in\mathbb S^{n-1},
 \qquad
 \lambda_1,\ldots,\lambda_m\ge0,
 \qquad
 \sum_{i=1}^m\lambda_i u_i u_i^{\top}=I_n.
\end{equation}
Labels remain distinct when their directions coincide, and weights may
vanish.  Let \(\DR(m,n,n)\) be the greatest volume of a simplex
\(\operatorname{conv}\{0,u_{i_1},\ldots,u_{i_n}\}\), with
\(1\le i_1<\cdots<i_n\le m\), guaranteed in every such decomposition,
and put \(d_n=\binom{n+1}{2}\).  In its original
volume formulation, their Conjecture~4.8 asserts that, for every \(n\ge1\),
\begin{equation}\label{eq:GMS-conjecture}
 \DR(d_n,n,n)=\frac{\sqrt{n+1}}{n!\,2^{n/2}}.
\end{equation}
The edge directions of a regular simplex form the \(A_n\) root
decomposition and attain this value, so the proposed bound would be sharp.

The geometric reason for this conjecture is that every convex body can
be put in \emph{Behrend position}, a linear image maximising its
isodiametric quotient.  In this position, its diametrical directions
admit a decomposition of the form \eqref{eq:identity-decomposition} with
at most \(d_n\) vectors.  A volume comparison for the corresponding
diameter segments then gives
\citep[Theorem~2.3 and Corollary~4.7]{GonzalezMerinoSchymura2020}
\[
 \operatorname{idq}(K)\ge\DR(d_n,n,n).
\]
Hence \cref{eq:GMS-conjecture} would imply Makai's conjectured bound.
This approach succeeds in dimension three, as \citet{Aliev2024} proves that
every six-vector unit decomposition of \(I_3\) contains a triple with
\[
 |\det(u_i,u_j,u_k)|\ge\frac1{\sqrt2},
\]
and uses this estimate to prove Makai's bound in \(\mathbb R^3\).
Together with the \(A_3\) example, his proof therefore gives
\(\DR(6,3,3)=1/(6\sqrt2)\), the three-dimensional case of
Conjecture~4.8.

Aliev's result raises the question of whether the same simplex-volume
guarantee holds in higher dimensions.  We show that it fails in all
sufficiently large dimensions, using decompositions built from random
frames.

To study these simplex volumes, we use the Gram matrix of the selected
vectors \(v_1,\ldots,v_n\), namely
\(\Gram(v_1,\ldots,v_n)=(\langle v_i,v_j\rangle)_{i,j=1}^n\).
Its determinant is \((n!V)^2\), where \(V\) is the volume of
\(\operatorname{conv}\{0,v_1,\ldots,v_n\}\).  Thus finding a large
simplex is equivalent to finding a large Gram determinant.  More
generally, we ask:
\nopagebreak[4]
\begin{center}
\begin{minipage}{0.88\linewidth}
\centering\itshape
For a unit decomposition with \(m\ge n\) vectors in dimension \(n\), how
large a Gram determinant can we always guarantee among its
\(n\)-element selections?
\end{minipage}
\end{center}

For an $n$-set $B\subseteq[m]$, where $[m]=\{1,\ldots,m\}$, put
\[
 c_B=\det\Gram(u_i:i\in B),
 \qquad
 \cM(u)=\max_{|B|=n}c_B.
\]
Define the optimal determinant guarantee by
\begin{equation}\label{eq:Dmn-definition}
 \cD(m,n)=\inf\{\cM(u):(u,\lambda)
                  \text{ satisfies \cref{eq:identity-decomposition}}\}.
\end{equation}
The classical Dvoretzky--Rogers lemma guarantees an $n$-element selection
with a large Gram determinant in every such decomposition
\citep{DvoretzkyRogers1950}.  A refinement due to
\citet[Proposition~2.1]{PelczynskiSzarek1991} accounts for the number of
labels: for $n\le m\le\binom{n+1}{2}$, it gives
\[
 \cD(m,n)\ge L_{m,n}:=\frac{m^n}{n^n\binom mn}.
\]
See also \citet[Theorem~5]{FodorNaszodiZarnocz2019} for a probabilistic
proof, and \citet[Theorem~1.5]{GonzalezMerinoSchymura2020} for the
corresponding simplex-volume bound.  We recall the elementary argument in
\cref{sec:universal}, where the displayed bound is proved for all $m\ge n$.

Fix $\alpha>1$ and set
\[
 m_n=\lfloor\alpha n\rfloor,
 \qquad
 f_\alpha=\left(1-\frac1\alpha\right)^{\alpha-1}.
\]
Stirling's formula gives \(L_{m_n,n}^{1/n}\to f_\alpha\).  We prove that
this lower bound is sharp at the exponential scale by constructing Haar
decompositions with the same rate.

For each $n$, let $Q_n$ consist of the first $n$ columns of a Haar-distributed
matrix in $O(m_n)$.  Its rows are nonzero almost surely.  Normalise each row
to unit length, and denote the resulting unit vectors by
$u_{n,1},\ldots,u_{n,m_n}$.  Define
\[
 M_n=\max_{\substack{B\subseteq[m_n]\\|B|=n}}
     \det\Gram(u_{n,i}:i\in B).
\]
Almost surely, the squared row lengths give positive weights satisfying
\cref{eq:identity-decomposition}, so \(\cD(m_n,n)\le M_n\).
We write \(\bbP_n\) for Haar probability on \(O(m_n)\), and omit the
subscript when the dimension is fixed.  For almost-sure statements, realise
the matrices on a common probability space with these Haar marginal laws.
Their dependence across dimensions is arbitrary.

\Needspace{10\baselineskip}
\begin{theorem}\label{thm:main}
For every fixed $\alpha>1$,
\[
 \cD(m_n,n)^{1/n}\longrightarrow f_\alpha.
\]
For any such common-space realisation, the normalised Haar construction
attains this exponential rate almost surely:
\[
 M_n^{1/n}\longrightarrow f_\alpha
 \qquad\text{almost surely}.
\]
\end{theorem}

In particular, for every \(b>f_\alpha\), all sufficiently large \(n\)
admit a positive-weight decomposition with \(m_n\) labels for which every
\(n\)-element selection has Gram determinant less than \(b^n\).

The simplex-volume guarantee introduced above is related to the
determinant guarantee by
\begin{equation}\label{eq:DR-D-relation}
 \DR(m,n,n)=\frac{\sqrt{\cD(m,n)}}{n!}.
\end{equation}
Consequently, \cref{thm:main} also gives
\[
 \bigl(n!\,\DR(\lfloor\alpha n\rfloor,n,n)\bigr)^{1/n}
 \longrightarrow \sqrt{f_\alpha}.
\]

Put \(\tau_n=(n+1)/2^n\).  By \cref{eq:DR-D-relation}, the
simplex-volume conjecture \eqref{eq:GMS-conjecture} is equivalent to
\[
 \cD(d_n,n)=\tau_n.
\]
In determinant terms, it asserts that the \(A_n\) root decomposition is
extremal, with maximal Gram determinant \(\tau_n\).

Here \(d_n/n\to\infty\), so the fixed-ratio theorem does not apply directly.
Splitting weights among parallel copies allows us to obtain the minimax
limit in this regime as well.

\begin{corollary}\label{cor:diverging-scale}
Let \(r_n\ge n\) be integers satisfying \(r_n/n\to\infty\).  Then
\[
 \cD(r_n,n)^{1/n}\longrightarrow e^{-1},
 \qquad
 \bigl(n!\,\DR(r_n,n,n)\bigr)^{1/n}\longrightarrow e^{-1/2}.
\]
Moreover, for every \(b>e^{-1}\), all sufficiently large \(n\) admit a
positive-weight \(r_n\)-label decomposition with \(\cM(u)<b^n\).
\end{corollary}

\begin{corollary}\label{cor:GMS}
With \(d_n=\binom{n+1}{2}\),
\[
 \cD(d_n,n)^{1/n}\longrightarrow e^{-1}<\frac12,
 \qquad
 \bigl(n!\,\DR(d_n,n,n)\bigr)^{1/n}
 \longrightarrow e^{-1/2}<\frac1{\sqrt2}.
\]
For all sufficiently large \(n\),
\[
 \DR\!\left(\binom{n+1}{2},n,n\right)
 <\frac{\sqrt{n+1}}{n!\,2^{n/2}}.
\]
\end{corollary}

\Cref{cor:GMS} disproves González Merino--Schymura Conjecture~4.8, a proposed
sufficient condition for Makai's reverse isodiametric conjecture
\citep{Makai1978,GonzalezMerinoSchymura2020}.
The failure of this sufficient condition does not by itself settle
Makai's conjecture.  A separate route is supplied by the strong-slicing
result of \citet[Theorem~1.1]{OpenAI2026StrongSlicing}: taking that result
as an external input, \cref{app:makai} derives Makai's inequality and its
simplex equality case.  This deduction does not use the false
simplex-selection assertion \eqref{eq:GMS-conjecture}.

We note that the diverging-redundancy limit may also be derived from the
construction of \citet[Theorem~5.1]{PelczynskiSzarek1991}.

The matching upper bound has two ingredients.  A uniform lower bound on
the row weights controls the effect of normalisation (\cref{sec:model}).
For a fixed selection, the determinant moments have an exponential rate
linear in the moment order (\cref{sec:moments}).  For any threshold above
\(f_\alpha\), choosing this order sufficiently large, but independently
of \(n\), makes a union bound over all selections decay exponentially.
Independence between selections is not needed.  \Cref{sec:maximum}
combines these summable estimates with the first Borel--Cantelli lemma to
prove \cref{thm:main}, and
\cref{sec:consequence} passes to simplex volumes and the conjectured label
count.

\section{The universal lower bound}\label{sec:universal}

We recall the lower bound from the introduction and give its short
Cauchy--Binet proof for all \(m\ge n\).
For \(B\subseteq[m]\), write
\(\lambda_B=\prod_{i\in B}\lambda_i\).

\begin{proposition}
\label{prop:universal-floor}
Every unit decomposition as in \cref{eq:identity-decomposition} satisfies
\begin{equation}\label{eq:universal-floor}
 \cM(u)\ge L_{m,n}.
\end{equation}
Consequently, if \(m_n/n\to\alpha>1\), then
\[
 L_{m_n,n}^{1/n}\longrightarrow f_\alpha.
\]
\end{proposition}

\begin{proof}
Let \(X\) be the \(n\times m\) matrix whose \(i\)-th column is
\(\sqrt{\lambda_i}u_i\).  Then \(XX^{\top}=I_n\).  Cauchy--Binet gives
\[
 1=\det(XX^{\top})
  =\sum_{|B|=n}\lambda_Bc_B
  \le \cM(u)e_n(\lambda_1,\ldots,\lambda_m).
\]
Here \(e_n\) is the elementary symmetric polynomial of degree \(n\).
Taking traces in \cref{eq:identity-decomposition} gives
$\sum_i\lambda_i=n$.  Maclaurin's inequality therefore yields
\[
 e_n(\lambda_1,\ldots,\lambda_m)
 \le\binom mn\left(\frac nm\right)^n,
\]
which proves \cref{eq:universal-floor}.

\Needspace{6\baselineskip}
Finally, with
\(h_\alpha:=\alpha\log\alpha-(\alpha-1)\log(\alpha-1)\), Stirling's
formula gives
\[
 \frac1n\log\binom{m_n}{n}\longrightarrow h_\alpha.
\]
Moreover, \(\alpha e^{-h_\alpha}=f_\alpha\) by definition.  This proves
the second assertion of \cref{prop:universal-floor}.
\end{proof}

\section{Constructing decompositions from Haar frames}\label{sec:model}

\subsection{Normalisation and the determinant identity}

For \(1\le n\le m\), write
\[
 \St_n(\mathbb R^m)
 =\{Q\in\mathbb R^{m\times n}:Q^{\top}Q=I_n\}.
\]
Let \(Q\in\St_n(\mathbb R^m)\) have rows \(x_i^{\top}\).  These rows
form a Parseval frame, and their squared lengths
\(\lambda_i=\lVert x_i\rVert^2\) are called row weights, or leverages.
We normalise them using the convention
\[
 u_i=
 \begin{cases}
  x_i/\lVert x_i\rVert,&x_i\ne0,\\
  e_1,&x_i=0,
 \end{cases}
\]
where \(e_1=(1,0,\ldots,0)\in\mathbb R^n\).  Since \(\lambda_i=0\) for a
zero row, the choice of \(u_i\) does not affect the decomposition.

\Needspace{6\baselineskip}
\begin{lemma}\label{lem:normalization}
With the notation above, every \(u_i\) is a unit vector and
\[
 \sum_{i=1}^m\lambda_i u_i u_i^{\top}=I_n,
 \qquad
 \sum_{i=1}^m\lambda_i=n.
\]
If every row weight is positive, then for every \(n\)-set \(B\),
\begin{equation}\label{eq:row-normalized-minor}
 c_B=\frac{\det(Q_B)^2}{\prod_{i\in B}\lambda_i},
\end{equation}
where \(Q_B\) is the square row submatrix indexed by \(B\).
\end{lemma}

\begin{proof}
By construction, \(\lVert u_i\rVert=1\) and
\(\lambda_i u_i u_i^{\top}=x_i x_i^{\top}\) for every \(i\).
Summing and using \(Q^{\top}Q=I_n\) gives the first identity; taking
traces gives the second.  When all row weights are positive, row
normalisation gives
\[
 \Gram(u_i:i\in B)
 =D_B^{-1/2}Q_BQ_B^{\top}D_B^{-1/2},
 \qquad
 D_B=\operatorname{diag}(\lambda_i:i\in B),
\]
which proves \cref{eq:row-normalized-minor}.
\end{proof}

\Needspace{14\baselineskip}
\subsection{A lower bound for all row weights}

The Haar frame \(Q_n\) from the introduction is uniformly distributed on
\(\St_n(\mathbb R^{m_n})\)
\citep[Sections~1.4.3 and~2.5.1]{Chikuse2003}.  Its row weights
\(\lambda_{n,i}\) have the same distribution and, by
\cref{lem:normalization}, sum to \(n\).
Thus each row weight has mean \(n/m_n\).  To obtain the sharp exponential
upper bound for \(M_n\) using \cref{eq:row-normalized-minor}, we need a
lower bound close to \(n/m_n\) that holds for every row.

\begin{lemma}\label{lem:leverage}
Fix \(\alpha>1\), put \(m_n=\lfloor\alpha n\rfloor\), and let
\(Q_n\) be Haar distributed on \(\St_n(\mathbb R^{m_n})\).  For every
\(0<\varepsilon<1\), there are constants \(c>0\) and \(N\), depending only
on \(\alpha\) and \(\varepsilon\), such that
\[
 \bbP\!\left\{
 \min_{1\le i\le m_n}\lambda_{n,i}
 <(1-\varepsilon)\frac n{m_n}
 \right\}
 \le 4m_n e^{-cn}
\]
for \(n\ge N\).  Moreover, every leverage is positive almost surely.
\end{lemma}

\begin{proof}
If \(m_n=n\), orthogonality gives \(\lambda_{n,i}=1\) for every \(i\), and
the estimate follows.  Increase \(N\), if necessary, so that
\(m_n>n\) for all \(n\ge N\).
Realise \(Q_n\) as the first \(n\) columns of a Haar orthogonal matrix.
A fixed row of that matrix is uniform on \(\mathbb S^{m_n-1}\), and
\(\lambda_{n,i}\) is the
sum of the squares of its first \(n\) coordinates.  Therefore
\[
 \lambda_{n,i}\ \stackrel{d}=\ \frac Z{Z+W},
 \qquad
 Z\sim\chi_n^2,
 \quad
 W\sim\chi_{m_n-n}^2,
\]
with \(Z\) and \(W\) independent.  Equivalently,
\(\lambda_{n,i}\sim\operatorname{Beta}(n/2,(m_n-n)/2)\).

Set \(\eta=\varepsilon/2\).  Then
\[
 \frac{n(1-\eta)}{n(1-\eta)+(m_n-n)(1+\eta)}
 \ge\frac n{m_n}\frac{1-\eta}{1+\eta}
 \ge(1-\varepsilon)\frac n{m_n}.
\]
Consequently,
\[
 \left\{\frac Z{Z+W}<(1-\varepsilon)\frac n{m_n}\right\}
 \subseteq
 \{Z<n(1-\eta)\}\cup\{W>(m_n-n)(1+\eta)\}.
\]
Since \(m_n/n\to\alpha>1\), the standard chi-square tail inequalities
bound the probability of the right side by \(4e^{-cn}\) for all
sufficiently large \(n\), with \(c>0\) depending only on \(\alpha\) and
\(\varepsilon\); see, for example,
\citet[Lemma~1, equations~(4.3)--(4.4)]{LaurentMassart2000}.  A union bound
over the \(m_n\) rows, using their common one-row marginal law, proves the
estimate in \cref{lem:leverage}.
The beta law is supported on \((0,1)\); a finite union of null events gives
simultaneous positivity of all leverages.
\end{proof}

\section{Determinant moments for a fixed selection}\label{sec:moments}

For the numerator in \cref{eq:row-normalized-minor}, we compute the
moments using the real matrix-beta integral in the normalisation of
\citet[Theorem~3.3.1 and Definition~3.3.2]{Muirhead1982}.

\begin{lemma}\label{lem:block-moment}
Let \(m\ge n\), let \(Q\) be Haar distributed on
\(\St_n(\mathbb R^m)\), and fix \(B\subseteq[m]\) with \(|B|=n\).  For
every positive integer \(s\),
\begin{equation}\label{eq:block-moment-factorial}
\begin{aligned}
 \bbE\det(Q_B)^{2s}
 &=\prod_{j=1}^n
   \frac{((n-j+1)/2)_s}{((m-j+1)/2)_s}
 \\
 &=\prod_{\ell=0}^{s-1}
   \frac{(n+2\ell)!\,(m-n+2\ell)!}
        {(2\ell)!\,(m+2\ell)!}.
\end{aligned}
\end{equation}
Here \((a)_s=a(a+1)\cdots(a+s-1)\).
\end{lemma}

\begin{proof}
If \(m=n\), then \(Q\) is orthogonal, the only \(n\)-row block is \(Q\)
itself, and all three expressions in the statement equal \(1\).  We first
suppose \(m\ge2n\).
Let \(G\in\mathbb R^{m\times n}\) have independent standard Gaussian
entries.  Its polar factor
\[
 Q=G(G^{\top}G)^{-1/2}
\]
is Haar distributed on the Stiefel manifold.  By row-exchange symmetry we
may take \(B=[n]\).  Partition \(G\) into its first \(n\) rows and the
remaining \(m-n\) rows, and put
\[
 W_1=G_B^{\top}G_B,
 \qquad
 W_2=G_{B^c}^{\top}G_{B^c}.
\]
These are independent real Wishart matrices with \(n\) and \(m-n\) degrees
of freedom, respectively.  Since \(m\ge2n\), both are positive definite
almost surely.  Set
\[
 S=W_1+W_2,
 \qquad
 U=S^{-1/2}W_1S^{-1/2}.
\]
Then \(0<U<I_n\), and
\[
 \det(Q_B)^2=\det(Q_B^{\top}Q_B)
 =\det U
 =\frac{\det W_1}{\det(W_1+W_2)}.
\]

For
\[
 \Gamma_n(a)=\pi^{n(n-1)/4}
              \prod_{j=1}^n\Gamma\!\left(a-\frac{j-1}{2}\right),
 \qquad
 B_n(a,b)=\frac{\Gamma_n(a)\Gamma_n(b)}{\Gamma_n(a+b)},
\]
the beta-I matrix with parameters \(a,b>(n-1)/2\) has density
\[
 \frac{1}{B_n(a,b)}
 (\det U)^{a-(n+1)/2}
 (\det(I_n-U))^{b-(n+1)/2},
 \qquad 0<U<I_n,
\]
with respect to Lebesgue measure on the real symmetric matrices.  The
Wishart-to-beta transformation gives \(U\) this law with
\(a=n/2\) and \(b=(m-n)/2\); the required parameter inequalities follow
from \(m\ge2n\).  Multiplying this density by
\((\det U)^s\) and integrating therefore gives
\[
 \bbE\det(Q_B)^{2s}
 =\frac{B_n(n/2+s,(m-n)/2)}{B_n(n/2,(m-n)/2)}.
\]
Expanding the multivariate gamma functions gives
the first product formula in \cref{eq:block-moment-factorial}.

To obtain the second product formula in \cref{eq:block-moment-factorial},
set \(k=n-j+1\) and interchange the two finite products:
\[
 \prod_{k=1}^n\prod_{\ell=0}^{s-1}
 \frac{k+2\ell}{m-n+k+2\ell}
 =\prod_{\ell=0}^{s-1}
 \frac{(n+2\ell)!}{(2\ell)!}
 \frac{(m-n+2\ell)!}{(m+2\ell)!}.
\]
This proves both formulas when \(m\ge2n\).

It remains to consider \(n<m<2n\).  Put \(r=m-n\), let \(O\) be Haar
distributed on \(O(m)\), let \(J=[n]\), and realise
\(Q=O_{[m],J}\).  Thus \(Q\) has the required Haar distribution, while
\(R=O_{[m],J^c}\) is Haar distributed on \(\St_r(\mathbb R^m)\).
Jacobi's complementary-minor identity states in this setting that
\[
 \det (O^{-1})_{J,B}
 =(-1)^{\sum_{j\in J}j+\sum_{i\in B}i}
   \frac{\det O_{B^c,J^c}}{\det O}.
\]
Since \((O^{-1})_{J,B}=(O_{B,J})^{\top}\) and
\(|\det O|=1\), it follows that
\[
 |\det Q_B|=|\det R_{B^c}|.
\]
The inequality \(m<2n\) gives \(m\ge2r\), so the matrix-beta case just
proved, applied to the \(r\times r\) block \(R_{B^c}\), yields
\begin{align*}
 \bbE\det(Q_B)^{2s}
 &=\bbE\det(R_{B^c})^{2s}\\
 &=\prod_{\ell=0}^{s-1}
   \frac{(r+2\ell)!\,(m-r+2\ell)!}
        {(2\ell)!\,(m+2\ell)!}\\
 &=\prod_{\ell=0}^{s-1}
   \frac{(n+2\ell)!\,(m-n+2\ell)!}
        {(2\ell)!\,(m+2\ell)!}.
\end{align*}
This is the second product formula in \cref{eq:block-moment-factorial}; the
same finite-product identity used above gives the first.
\end{proof}

Note that at \(s=1\), \cref{eq:block-moment-factorial} reduces to
\[
 \bbE\det(Q_B)^2=\frac{n!(m-n)!}{m!}=\binom mn^{-1}.
\]
This is also forced by row-exchange symmetry and Cauchy--Binet, since
\(\sum_{|B|=n}\det(Q_B)^2=1\).

For fixed \(s\), applying Stirling's formula to
\cref{eq:block-moment-factorial} gives the following estimate.

\begin{lemma}\label{lem:entropy}
Fix \(\alpha>1\), put \(m_n=\lfloor\alpha n\rfloor\), and let \(Q_n\)
be Haar distributed on \(\St_n(\mathbb R^{m_n})\).  For each \(n\), fix
\(B_n\subseteq[m_n]\) with \(|B_n|=n\), and fix a positive integer \(s\).
Then
\[
 \frac1n\log\bbE\det(Q_{n,B_n})^{2s}=-s h_\alpha+o(1).
\]
\end{lemma}

\begin{proof}
\Cref{lem:block-moment} applies because
\(m_n\ge n\).  Moreover,
\[
 m_n=\alpha n+O(1),
 \qquad
 m_n-n=(\alpha-1)n+O(1),
\]
and both linear coefficients are positive.
\Needspace{5\baselineskip}
For each fixed \(\ell\), apply
\[
 \log\Gamma(cn+O(1))
 =cn(\log n+\log c-1)+O(\log n)
\]
to the three factorials in each quotient whose arguments grow linearly with
\(n\).  Since \(\ell\) is fixed, the factor \((2\ell)!\) contributes
\(o(n)\) to the logarithm.  Thus, for every fixed \(\ell\), division by
\(n\) gives
\[
 \frac1n\log
 \frac{(n+2\ell)!\,(m_n-n+2\ell)!}
      {(2\ell)!\,(m_n+2\ell)!}
 =-h_\alpha+o(1).
\]
Summing over the fixed set \(\ell=0,\ldots,s-1\) proves
\cref{lem:entropy}.
\end{proof}

\section{Proof of the main theorem}\label{sec:maximum}

The lower bound in \cref{thm:main} follows from \cref{sec:universal}.
To obtain the matching upper bound, we must control all
\(\binom{m_n}{n}\) selections simultaneously. Outside the exceptional
event in \cref{lem:leverage}, all row weights are bounded below, so
\cref{eq:row-normalized-minor} allows us to apply the moment estimate of
\cref{lem:entropy} to the normalised determinants. For each
\(b>f_\alpha\), a sufficiently large fixed moment order makes
the probability bound for a single selection small enough that its sum
over all selections is summable in \(n\).

\begin{proposition}\label{prop:upper-tail}
Under the hypotheses of \cref{thm:main}, for every \(b>f_\alpha\), there
are constants \(C,\beta>0\) and \(N\), depending only on \(\alpha,b\), such
that
\[
 \bbP_n\{M_n\ge b^n\}\le Ce^{-\beta n}
 \qquad(n\ge N).
\]
In particular,
\[
 \sum_{n=1}^{\infty}\bbP_n\{M_n\ge b^n\}<\infty.
\]
\end{proposition}

\begin{proof}
Choose \(0<\varepsilon<1\) so small that
\(b(1-\varepsilon)>f_\alpha\), and let \(\mathcal L_{n,\varepsilon}\) be
the event that
\(\lambda_{n,i}\ge(1-\varepsilon)n/m_n\) for every \(i\).
By \cref{eq:row-normalized-minor}, on this event the inequality
\(c_B\ge b^n\) implies
\[
 \det(Q_{n,B})^2
 \ge\left(b(1-\varepsilon)\frac n{m_n}\right)^n.
\]
For every positive integer \(s\), the preceding implication and Markov's
inequality give
\[
\begin{aligned}
 \bbP\bigl(\{c_B\ge b^n\}\cap\mathcal L_{n,\varepsilon}\bigr)
 &\le \bbP\!\left\{
  \det(Q_{n,B})^2
  \ge [b(1-\varepsilon)n/m_n]^n
 \right\}\\
 &\le
 \frac{\bbE\det(Q_{n,B})^{2s}}
      {[b(1-\varepsilon)n/m_n]^{ns}}.
\end{aligned}
\]
By row-exchange symmetry, the right-hand side is the same for every
selection \(B\). A union bound over the \(\binom{m_n}{n}\) selections,
together with the probability that \(\mathcal L_{n,\varepsilon}\) fails,
therefore gives
\begin{equation}\label{eq:master-union-bound}
 \bbP\{M_n\ge b^n\}
 \le \bbP(\mathcal L_{n,\varepsilon}^{c})
 +\binom{m_n}{n}
  \frac{\bbE\det(Q_{n,B})^{2s}}
       {[b(1-\varepsilon)n/m_n]^{ns}}.
\end{equation}
Write \(T_n\) for the second term on the right of
\cref{eq:master-union-bound}.  By \cref{lem:entropy} and
\(m_n/n\to\alpha\),
\begin{equation}\label{eq:union-rate}
 \frac1n\log T_n
 =h_\alpha+s\log\frac{f_\alpha}{b(1-\varepsilon)}+o(1).
\end{equation}
Because \(b(1-\varepsilon)>f_\alpha\), one has
\[
 \log\frac{f_\alpha}{b(1-\varepsilon)}<0.
\]
Choose a fixed integer
\[
 s>\frac{h_\alpha}
          {\log\!\bigl(b(1-\varepsilon)/f_\alpha\bigr)}.
\]
Put
\[
 \delta=s\log\!\frac{b(1-\varepsilon)}{f_\alpha}-h_\alpha>0.
\]
By \cref{eq:union-rate}, \(T_n\le e^{-\delta n/2}\) for all sufficiently
large \(n\).  Since \(m_n\le\alpha n\), \cref{lem:leverage} and
\cref{eq:master-union-bound} give
\[
 \bbP_n\{M_n\ge b^n\}
 \le4\alpha n e^{-cn}+e^{-\delta n/2}
 \le Ce^{-\beta n}
\]
eventually, where
\[
 \beta=\tfrac12\min(c,\delta)>0,
 \qquad C=1+\frac{8\alpha}{ce}.
\]
The last inequality uses \(xe^{-cx/2}\le2/(ce)\) for \(x\ge0\).
The geometric tail is summable, and the finitely many earlier terms do
not affect convergence of the series.
\end{proof}

\begin{proof}[Proof of \cref{thm:main}]
By \cref{lem:normalization}, the row-normalised family is always an
admissible decomposition; by \cref{lem:leverage}, all its weights are
positive almost surely.  The deterministic lower bound in
\cref{prop:universal-floor} gives
\[
 M_n^{1/n}\ge L_{m_n,n}^{1/n}\longrightarrow f_\alpha.
\]
Let \(\bbP\) now denote the probability on any common-space realisation
of the Haar matrices.  For every fixed \(b>f_\alpha\), the marginal laws
and \cref{prop:upper-tail} give
\[
 \sum_{n=1}^{\infty}\bbP\{M_n\ge b^n\}
 =\sum_{n=1}^{\infty}\bbP_n\{M_n\ge b^n\}<\infty.
\]
The first Borel--Cantelli lemma implies that \(M_n^{1/n}<b\) eventually
almost surely.  Apply this to \(b_k=f_\alpha+1/k\), \(k\ge1\), and
intersect the resulting countably many probability-one events.  Then
\(\limsup_{n\to\infty}M_n^{1/n}\le f_\alpha\) almost surely.  The
deterministic lower bound above proves the random limit in
\cref{thm:main}.  No independence across dimensions is required.

The same lower bound gives
\[
 \liminf_{n\to\infty}\cD(m_n,n)^{1/n}\ge f_\alpha.
\]
Conversely, for every \(b>f_\alpha\), \cref{prop:upper-tail} and the
almost-sure positivity of the leverages imply that a positive-weight frame
with \(\cM(u)<b^n\) exists for all sufficiently large \(n\).
Therefore
\[
 \limsup_{n\to\infty}\cD(m_n,n)^{1/n}\le b.
\]
Letting \(b\downarrow f_\alpha\) proves
the deterministic limit in \cref{thm:main}.
\end{proof}

When \(\alpha>2\), even linearly many labels suffice to beat
\(\tau_n=(n+1)/2^n\).

\Needspace{6\baselineskip}
\begin{corollary}\label{cor:candidate-scale}
If \(\alpha>2\), then
\[
 \bbP_n\!\left\{M_n<\frac{n+1}{2^n}\right\}\longrightarrow1.
\]
Consequently, for all sufficiently large \(n\), there is a deterministic
positive-weight decomposition with \(\lfloor\alpha n\rfloor\) labels and
\(\cM(u)<(n+1)/2^n\).
\end{corollary}

\Needspace{9\baselineskip}
\begin{proof}
Differentiation gives
\[
 \frac{d}{d\alpha}\log f_\alpha
 =\log\left(1-\frac1\alpha\right)+\frac1\alpha<0.
\]
Thus \(f_\alpha<f_2=1/2\) for \(\alpha>2\).  Choose
\(f_\alpha<b<1/2\).  Since
\(\tau_n^{1/n}=(n+1)^{1/n}/2\to1/2\), we have \(b^n<\tau_n\) for all
large \(n\).  The probability assertion follows from
\cref{prop:upper-tail}; for large \(n\), this event has positive probability,
giving the deterministic statement.
\end{proof}

At \(\alpha=2\), the exponential base reaches the candidate value
\(f_2=1/2\). Resolving the comparison with \(\tau_n\) requires better
estimates.

\section{Consequences for simplex volumes}
\label{sec:consequence}

For \(1\le j\le n\), let \(\operatorname{vol}_j\) be the normalised
\(j\)-dimensional Hausdorff measure that agrees with ordinary Lebesgue
measure on every \(j\)-dimensional affine subspace.  Following
\citet[Definition~4.2]{GonzalezMerinoSchymura2020}, define
\begin{equation}\label{eq:DR-definition}
\begin{aligned}
 \DR(m,n,j)=\sup\{\nu\ge0: {}&
  \forall (u,\lambda)\text{ satisfying \cref{eq:identity-decomposition}},\\
 &\exists B\subseteq[m],\quad |B|=j,\quad\text{such that}\\[-1mm]
 &\operatorname{vol}_j\operatorname{conv}
  (\{0\}\cup\{u_i:i\in B\})\ge\nu\}.
\end{aligned}
\end{equation}
\Needspace{8\baselineskip}
\begin{proposition}\label{prop:DR-D-bridge}
For \(1\le n\le m\), \(\DR(m,n,n)\) satisfies
\cref{eq:DR-D-relation}.
\end{proposition}

\begin{proof}
For every \(n\)-set \(B\),
\[
 \operatorname{vol}_n\operatorname{conv}(\{0\}\cup\{u_i:i\in B\})
 =\frac{|\det(u_i:i\in B)|}{n!}
 =\frac{\sqrt{c_B}}{n!}.
\]
Taking maxima over \(B\), and then infima over admissible decompositions,
gives \cref{eq:DR-D-relation}.
\end{proof}

\subsection{Adding labels without changing the maximum}

We pad a decomposition to a larger number of labels by splitting weights
among parallel copies, preserving its maximal determinant.

\begin{lemma}\label{lem:padding}
Suppose \cref{eq:identity-decomposition} has \(m\le M\) labels.  There is an
\(M\)-label unit decomposition of the identity, with positive weights if
the original weights are positive, whose maximal Gram determinant is
exactly \(\cM(u)\).
\end{lemma}

\begin{proof}
Taking traces shows that at least one weight, say \(\lambda_{i_0}\), is
positive.  Put \(k=M-m+1\), choose positive numbers
\(\lambda_{i_0,1},\ldots,\lambda_{i_0,k}\) summing to
\(\lambda_{i_0}\), and replace the label \(i_0\) by \(k\) parallel copies:
\[
 \lambda_{i_0,1}u_{i_0}u_{i_0}^{\top}+\cdots+
 \lambda_{i_0,k}u_{i_0}u_{i_0}^{\top},
 \qquad
 \sum_{a=1}^k\lambda_{i_0,a}=\lambda_{i_0}.
\]
The identity decomposition is unchanged.  Let \(\pi\) map each new label to
its original label.  If \(\pi\) is not injective on an \(n\)-set of new
labels, that set contains two parallel copies and its Gram determinant is
zero.  If \(\pi\) is injective, its Gram determinant is the determinant of
the original \(n\)-set \(\pi(B)\).  Conversely, every original \(n\)-set is
recovered by choosing one representative above each of its labels.  Hence
the maximum is unchanged.  If every original weight is positive, every new
weight is positive as well.
\end{proof}

\subsection{The limit for diverging redundancy}

Padding allows us to apply \cref{thm:main} with \(\alpha\) held fixed as
\(n\to\infty\).  We can then let \(\alpha\to\infty\).

\begin{proof}[Proof of \cref{cor:diverging-scale}]
\Cref{lem:padding} implies the monotonicity
\[
 \cD(M,n)\le\cD(m,n)\qquad(n\le m\le M).
\]
Every admissible \(m\)-label decomposition has an \(M\)-label padding with
the same maximum; taking infima gives the displayed monotonicity.

For the lower bound, \(\binom{r_n}{n}\le r_n^n/n!\), so
\cref{prop:universal-floor} gives
\[
 \cD(r_n,n)\ge\frac{n!}{n^n}.
\]
Stirling's formula shows that the \(n\)-th root of the right side tends to
\(e^{-1}\).

For the upper bound, fix \(\alpha>1\).  Since
\(\lfloor\alpha n\rfloor\le r_n\) for all sufficiently large \(n\), padding
monotonicity and \cref{thm:main} yield
\[
 \limsup_{n\to\infty}\cD(r_n,n)^{1/n}\le f_\alpha.
\]
Moreover,
\[
 \log f_\alpha
 = (\alpha-1)\log\left(1-\frac1\alpha\right)
 \longrightarrow-1
 \qquad(\alpha\to\infty).
\]
Letting \(\alpha\to\infty\) gives the matching upper limit \(e^{-1}\) and
proves the determinant limit in \cref{cor:diverging-scale}.  By
\cref{eq:DR-D-relation},
\[
 \bigl(n!\,\DR(r_n,n,n)\bigr)^{1/n}
 =\bigl(\cD(r_n,n)^{1/n}\bigr)^{1/2},
\]
which proves the second limit.

Now let \(b>e^{-1}\).  Choose a sufficiently large \(\alpha\ge2\) and
\(b'\) such that \(f_\alpha<b'<b\).  By \cref{prop:upper-tail} and the
almost-sure positivity in \cref{lem:leverage}, all sufficiently large
\(n\) admit a positive-weight
\(\lfloor\alpha n\rfloor\)-label decomposition with
\(\cM(u)<(b')^n<b^n\).  Since \(\lfloor\alpha n\rfloor\le r_n\) eventually,
padding it to \(r_n\) labels provides the required positive-weight
decomposition.
\end{proof}

\subsection{Failure of Gonz\'alez Merino--Schymura Conjecture~4.8}

The value in \cref{eq:GMS-conjecture} comes from the \(A_n\) root
decomposition.  In \(\boldsymbol1^\perp\subseteq\mathbb R^{n+1}\), take
\(u_{ij}=(e_i-e_j)/\sqrt2\), \(1\le i<j\le n+1\), with weight
\(2/(n+1)\).  Its independent
\(n\)-sets are the spanning trees of \(K_{n+1}\), and the matrix--tree
theorem gives Gram determinant \((n+1)/2^n\) for each such set.  Hence
\(\cD(d_n,n)\le\tau_n\); the conjecture asserts equality.
The diverging-redundancy limit now gives a strictly smaller exponential
rate.

\begin{proof}[Proof of \cref{cor:GMS}]
Since \(d_n/n=(n+1)/2\to\infty\), taking \(r_n=d_n\) proves
the two limits in \cref{cor:GMS} by \cref{cor:diverging-scale}.

Moreover, \(\tau_n^{1/n}=(n+1)^{1/n}/2\to1/2\), whereas
\(\cD(d_n,n)^{1/n}\to e^{-1}<1/2\).  Hence
\(\cD(d_n,n)<\tau_n\) for all sufficiently large \(n\).  The bridge in
\cref{prop:DR-D-bridge} now gives the strict inequality in \cref{cor:GMS}.
\end{proof}

\Cref{cor:GMS} shows that Gonz\'alez Merino--Schymura Conjecture~4.8 fails
in every sufficiently large dimension.

\clearpage
\phantomsection

\clearpage
\appendix
\section{Makai's conjecture from OpenAI's strong-slicing result}
\label[appendix]{app:makai}

We derive Makai's reverse isodiametric conjecture from the strong-slicing
theorem of \citet[Theorem~1.1]{OpenAI2026StrongSlicing}.  After making the
covariance matrix the identity, Izmestiev's width--variance estimate
\citep[Lemma~6.4]{Izmestiev2014} bounds the diameter, while the isotropic
constant determines the volume.  This gives the required affine
position.  We then relate the resulting bound to a determinant formed
from the chords of the body.

\subsection{Strong slicing and Makai's conjecture}

Fix \(n\ge1\).  Throughout this appendix, a convex body is a compact
convex subset of \(\mathbb R^n\) with nonempty interior.  Write
\(V(K)=\operatorname{vol}_n(K)\), \(D(K)=\operatorname{diam}(K)\), and
\[
 \mathcal B_n(K):=\sup_{A\in\operatorname{GL}_n(\mathbb R)}
                  \frac{V(AK)}{D(AK)^n},
 \qquad
 c_n:=\frac{\sqrt{n+1}}{2^{n/2}n!}.
\]
Translations leave volume and diameter unchanged, so \(\mathcal B_n\)
is invariant under invertible affine maps.

Let \(b_K\) and \(\Sigma_K\) denote the centroid and covariance matrix
of the uniform probability measure on \(K\).  Thus
\[
 b_K=\frac1{V(K)}\int_K x\,dx,
 \qquad
 \Sigma_K=\frac1{V(K)}\int_K(x-b_K)(x-b_K)^{\top}\,dx.
\]
The isotropic constant is defined by
\[
 L_K^{2n}:=\frac{\det\Sigma_K}{V(K)^2}.
\]

Let \(\Delta_n\) be an \(n\)-simplex.  The strong-slicing theorem stated
in \citet[Theorem~1.1]{OpenAI2026StrongSlicing} gives
\begin{equation}\label{eq:strong-slicing-input}
 L_K\le L_{\Delta_n}
 :=\frac{(n!)^{1/n}}{(n+1)^{(n+1)/(2n)}\sqrt{n+2}},
\end{equation}
with equality if and only if \(K\) is a simplex.  We use this bound to
obtain the following form of Makai's conjecture.

\begin{theorem}[Makai from Strong Slicing]\label{thm:makai-from-slicing}
Assume \eqref{eq:strong-slicing-input} together with its equality
statement.  Then every convex body
\(K\subset\mathbb R^n\) admits an invertible linear map \(A\) such that
\[
 \frac{V(AK)}{D(AK)^n}\ge c_n.
\]
Moreover, \(\mathcal B_n(K)=c_n\) if and only if \(K\) is a simplex.
\end{theorem}

\subsection{Covariance normalisation}

We first control the diameter in covariance-normalised position.  For
a unit vector \(u\), write
\[
 w_K(u):=\max_{x\in K}\langle x,u\rangle
          -\min_{x\in K}\langle x,u\rangle.
\]
Izmestiev's estimate \citep[Lemma~6.4]{Izmestiev2014} gives
\begin{equation}\label{eq:makai-width-variance}
 u^{\top}\Sigma_Ku\ge\frac{w_K(u)^2}{2(n+1)(n+2)}.
\end{equation}
For \(n=1\), this is the variance formula for the uniform distribution
on an interval.

\begin{proposition}[Covariance Normalisation]\label{prop:makai-covariance}
Every convex body \(K\subset\mathbb R^n\) admits an invertible affine
map \(T\) such that
\[
 D(TK)\le1,
 \qquad
 V(TK)=c_n\left(\frac{L_{\Delta_n}}{L_K}\right)^n.
\]
Consequently,
\begin{equation}\label{eq:makai-covariance-bound}
 \mathcal B_n(K)\ge
 c_n\left(\frac{L_{\Delta_n}}{L_K}\right)^n.
\end{equation}
\end{proposition}

\begin{proof}
Since \(K\) has nonempty interior, \(\Sigma_K\) is positive definite.
Put \(a_n=2(n+1)(n+2)\) and
\(E=\Sigma_K^{-1/2}(K-b_K)\).  A change of variables gives
\[
 \operatorname{Cov}(\operatorname{Unif}(E))=I_n,
 \qquad
 V(E)=\frac{V(K)}{\sqrt{\det\Sigma_K}}=L_K^{-n}.
\]
The covariance identity turns \eqref{eq:makai-width-variance} into a
uniform bound on the widths of \(E\).  For distinct \(x,y\in E\), take
\(u=(x-y)/\|x-y\|\).  Then
\(\|x-y\|\le w_E(u)\), and \eqref{eq:makai-width-variance} gives
\(w_E(u)^2\le a_n\).  Hence \(D(E)\le\sqrt{a_n}\).

We now rescale \(E\) to make its diameter at most one.  The affine map
\[
 T(x)=a_n^{-1/2}\Sigma_K^{-1/2}(x-b_K)
\]
therefore satisfies \(D(TK)\le1\) and
\[
 V(TK)=a_n^{-n/2}L_K^{-n}
      =c_n\left(\frac{L_{\Delta_n}}{L_K}\right)^n,
\]
where we have used
\[
 a_n^{n/2}L_{\Delta_n}^{\,n}
 =\frac{2^{n/2}n!}{\sqrt{n+1}}=c_n^{-1}.
\]
Translations do not affect volume or diameter, so the linear part
\(A=a_n^{-1/2}\Sigma_K^{-1/2}\) satisfies
\[
 \mathcal B_n(K)\ge\frac{V(AK)}{D(AK)^n}
 =\frac{V(TK)}{D(TK)^n}\ge V(TK).
\]
This proves \eqref{eq:makai-covariance-bound}.
\end{proof}

To settle the equality case, we need the maximal volume of a simplex
of fixed diameter.  The sum of the outer products of the centred
vertices gives this bound: its trace is controlled by the pairwise
distances, and its determinant is a constant multiple of the squared
volume.

\begin{lemma}[The Simplex Case]\label{lem:makai-simplex}
Every \(n\)-simplex \(S\subset\mathbb R^n\) satisfies
\(V(S)\le c_nD(S)^n\), and \(\mathcal B_n(S)=c_n\).
\end{lemma}

\begin{proof}
Let \(p_0,\ldots,p_n\) be the vertices of \(S\), and put
\(v_i=p_i-(\sum_jp_j)/(n+1)\) and
\(G=\sum_i v_iv_i^{\top}\).  The identity \(\sum_i v_i=0\) gives
\[
 (n+1)\operatorname{tr}G=\sum_{i<j}\|p_i-p_j\|^2
 \le\frac{n(n+1)}2D(S)^2.
\]
To express the determinant in terms of volume, let \(Y\) have columns
\(p_1-p_0,\ldots,p_n-p_0\), and let
\(\mathbf1\in\mathbb R^n\) be the all-ones vector.  Then
\[
 G=Y\left(I_n-\frac{\mathbf1\mathbf1^{\top}}{n+1}\right)Y^{\top},
 \qquad
 \det G=\frac{(n!)^2}{n+1}V(S)^2.
\]
The arithmetic--geometric mean inequality for the eigenvalues of \(G\)
now yields
\[
 V(S)^2=\frac{n+1}{(n!)^2}\det G
 \le\frac{n+1}{(n!)^2}
       \left(\frac{\operatorname{tr}G}{n}\right)^n
 \le c_n^2D(S)^{2n}.
\]
This proves the volume bound.  Every invertible affine image of \(S\)
is again a simplex, so \(\mathcal B_n(S)\le c_n\).  A regular simplex
of edge length one has \(G=I_n/2\) and volume \(c_n\).  Since every
\(n\)-simplex is affinely equivalent to it, the reverse inequality
also holds.
\end{proof}

\begin{proof}[Proof of \cref{thm:makai-from-slicing}]
The linear map constructed in \cref{prop:makai-covariance} satisfies
\[
 \frac{V(AK)}{D(AK)^n}\ge
 c_n\left(\frac{L_{\Delta_n}}{L_K}\right)^n\ge c_n,
\]
where the last inequality follows from \eqref{eq:strong-slicing-input}.
If \(K\) is not a simplex, the equality statement in the strong-slicing
theorem gives \(L_K<L_{\Delta_n}\), and hence
\(\mathcal B_n(K)>c_n\).  Conversely, \cref{lem:makai-simplex} gives
\(\mathcal B_n(K)=c_n\) for every simplex.
\end{proof}

\subsection{The chord determinant}

The quotient \(\mathcal B_n(K)\) also has a determinant description.
Instead of choosing an affine image, we maximise the determinant of a
weighted Gram matrix formed from the chords of \(K\).  Define
\[
 \mathcal Q_n(K):=
 \left\{\sum_{j=1}^m\lambda_j u_ju_j^{\top}:
 m\ge1,\ u_j\in K-K,\ \lambda_j\ge0,\ \sum_j\lambda_j=n\right\}.
\]
Set
\[
 \kappa_n(K):=\max_{Q\in\mathcal Q_n(K)}\det Q,
 \qquad
 F_n(K):=\sqrt{\kappa_n(K)}.
\]
Since \(K-K\) is compact, so is
\(\mathcal Q_n(K)=n\operatorname{conv}\{uu^{\top}:u\in K-K\}\).
The nonempty interior of \(K\) supplies \(n\) linearly independent
chords.  Thus \(\mathcal Q_n(K)\) contains a positive definite matrix,
and \(F_n(K)>0\).

\begin{proposition}[Chord Determinant Formula]\label{prop:makai-chord-determinant}
For every convex body \(K\subset\mathbb R^n\),
\[
 \mathcal B_n(K)=\frac{V(K)}{F_n(K)}.
\]
\end{proposition}

\begin{proof}
We first bound \(\mathcal B_n(K)\) from above.  Each
\(Q\in\mathcal Q_n(K)\) is positive semidefinite and satisfies
\(\operatorname{tr}Q\le nD(K)^2\).  Thus
\(\det Q\le(\operatorname{tr}Q/n)^n\le D(K)^{2n}\), giving
\(F_n(K)\le D(K)^n\).  For any invertible linear map \(A\),
\[
 \mathcal Q_n(AK)=A\mathcal Q_n(K)A^{\top},
 \qquad F_n(AK)=|\det A|F_n(K).
\]
Volume transforms by the same factor, so
\[
 \frac{V(AK)}{D(AK)^n}
 \le\frac{V(AK)}{F_n(AK)}=\frac{V(K)}{F_n(K)}.
\]
Taking the supremum over \(A\) proves
\(\mathcal B_n(K)\le V(K)/F_n(K)\).

For the reverse inequality, let \(Q_*\) maximise the determinant on
\(\mathcal Q_n(K)\).  Since \(F_n(K)>0\), the matrix \(Q_*\) is
positive definite.  Fix \(u\in K-K\).  The segment from \(Q_*\) to
\(nuu^{\top}\) lies in \(\mathcal Q_n(K)\), so maximality makes the
one-sided derivative of \(\log\det\) along this segment nonpositive.
Therefore
\[
 0\ge\operatorname{tr}\bigl(Q_*^{-1}(nuu^{\top}-Q_*)\bigr)
   =n\bigl(u^{\top}Q_*^{-1}u-1\bigr).
\]
It follows that every chord of \(Q_*^{-1/2}K\) has length at most one.
Thus
\[
 D(Q_*^{-1/2}K)\le1,
 \qquad
 V(Q_*^{-1/2}K)=\frac{V(K)}{F_n(K)}.
\]
Choosing \(A=Q_*^{-1/2}\) in the definition of \(\mathcal B_n(K)\)
proves the reverse inequality.
\end{proof}

Under the hypotheses of \cref{thm:makai-from-slicing},
\cref{prop:makai-chord-determinant} therefore gives
\[
 V(K)^2\ge\frac{n+1}{2^n(n!)^2}\kappa_n(K),
\]
with equality if and only if \(K\) is a simplex.  This formulation uses
all chords of \(K\), whereas \eqref{eq:GMS-conjecture} concerns
selections from a prescribed unit decomposition.

\end{document}